\documentclass
{amsart}
\usepackage[utf8]{inputenc}
\usepackage{amssymb,amscd}
\usepackage[all,line,arc,curve,color,frame,pdf]{xy}
\usepackage{tikz}
\usepackage{tikz-cd}
\usetikzlibrary{positioning, trees, snakes}
\usepackage[textsize=tiny]{todonotes}
\usepackage{hyperref}
\usepackage[shortlabels]{enumitem}
\usepackage{float} 
\usepackage[paper=a4paper, margin=3.5cm]{geometry}
\usepackage{cleveref}   
\usepackage{comment}
\usepackage{scalerel}
\usepackage[normalem]{ulem}
\usepackage{listings}
\usepackage{verbatim}
\usepackage{marvosym} 

\numberwithin{equation}{section}
\theoremstyle{plain}
\newtheorem{thm}{Theorem}[section]
\newtheorem{cor}[thm]{Corollary}
\newtheorem{lemma}[thm]{Lemma}
\newtheorem{prop}[thm]{Proposition}

\newtheorem{remark}[thm]{Remark}

\theoremstyle{remark}

\newtheorem*{notation*}{Notation}

\theoremstyle{definition}
\newtheorem{defi}[thm]{Definition}

\def\ot{\otimes}

\newcommand{\HS}{{\rm HS}}

\usepackage[textsize=tiny]{todonotes}

\usepackage{enumitem}

\newcommand\ignore[1]{}

\def\C{{\mathbb C}}
\def\F{{\mathbb F}}

\def\N{{\mathbb N}}
\def\P{{\mathbb P}}

\def\R{{\mathbb R}}

\def\operatorname#1{\mathop{\rm #1}\nolimits}

\def\deg{\operatorname{deg}}

\newcommand{\pb}{\ar@{}[dr]|{\text{\pigpenfont J}}}

\makeatletter
\newcommand{\xleftrightarrow}[2][]{\ext@arrow 3359\leftrightarrowfill@{#1}{#2}}
\makeatother
\newcommand{\xdasharrow}[2][->]{
\tikz[baseline=-\the\dimexpr\fontdimen22\textfont2\relax]{
\node[anchor=south,font=\scriptsize, inner ysep=1.5pt,outer xsep=2.2pt](x){#2};
\draw[shorten <=3.4pt,shorten >=3.4pt,dashed,#1](x.south west)--(x.south east);
}}

\newcommand{\aff}[1]{\widehat{#1}}

\begin{document}
\title[Identifiability for mixtures of centered Gaussians]{Identifiability for mixtures of centered Gaussians and sums of powers of quadratics}

\author[Blomenhofer]{Alexander Taveira Blomenhofer}
\address{University of Konstanz, Germany, Fachbereich Mathematik und Statistik, 
	D-78457 Konstanz, Germany}
\email{alexander.taveira-blomenhofer@uni-konstanz.de}

\author[Casarotti]{Alex Casarotti}
\address{Università di Trento, Via Sommarive, 14 - 38123 Povo (Trento), Italy}
\email{alex.casarotti@unitn.it}

\author[Micha{\l}ek]{Mateusz Micha{\l}ek}
\address{
	University of Konstanz, Germany, Fachbereich Mathematik und Statistik, Fach D 197
	D-78457 Konstanz, Germany
}
\thanks{MM is funded by the Deutsche Forschungsgemeinschaft –- Projektnummer 467575307.}
\email{mateusz.michalek@uni-konstanz.de}

\author[Oneto]{Alessandro Oneto}
\address{Università di Trento, Via Sommarive, 14 - 38123 Povo (Trento), Italy}
\email{alessandro.oneto@unitn.it}
\thanks{AO is partially supported by MIUR and GNSAGA of INdAM (Italy).}

\begin{abstract} 
	We consider the inverse problem for the polynomial map which sends an $m$-tuple of quadratic forms in $n$ variables to the sum of their $d$-th powers. This map captures the moment problem for \emph{mixtures of $m$ centered $n$-variate Gaussians}. In the first non-trivial case $d = 3$, we show that for any $ n\in \N $, this map is generically one-to-one (up to permutations of $ q_1,\ldots, q_m $ and third roots of unity) in two ranges: $m\le {n\choose 2} + 1 $ for $n \le 16$  and  $  m\le {n+5 \choose 6}/{n+1 \choose 2}-{n+1 \choose 2}-1$ for $n > 16$, thus proving \emph{generic identifiability} for mixtures of centered Gaussians from their (exact) moments of degree at most $ 6 $. The first result is obtained by the explicit geometry of the tangential contact locus of the variety of sums of cubes of quadratic forms, as described in \cite{Chiantini_Ottaviani_2012}, while the second result is accomplished using the link between secant non-defectivity with identifiability, proved in \cite{casarotti2022non}. The latter approach generalizes also to sums of $ d $-th powers of $k$-forms for $d \geq 3$ and $k \geq 2$. 
\end{abstract}

\maketitle

\section{Introduction}

\subsection{Motivation} 

Ever since the pioneering work of Pearson (original work in \cite{Pearson_1900}, but see e.g. \cite{Amendola_Faugere_Sturmfels_2016} for a modern exposition) on the separation of biological species, \emph{Gaussian mixture models} are a highly important tool for modern data analysis. Given the moments of a mixture $ Y $ of $n$-dimensional normally distributed random variables $ Y_i \sim \mathcal{N}(\mu_i, \Sigma_i) $, one aims to recover the parameters $ (\mu_1, \Sigma_1), \ldots, (\mu_m, \Sigma_m) $ (up to permutation). One of the first and fundamental questions to ask is under which circumstances there exists a unique solution, essentially justifying that the parameters bear information which is ``meaningful'' for statistical inference. In this work, we are concerned with the special case of centered Gaussians, i.e., $ \mu_1 = \ldots = \mu_m = {0} \in \R^n$, and we focus on moments of one, fixed degree. Then, the estimation problem turns out to be profoundly related to an algebraic problem. Identify $S^d(\C^n)$ with the vector space of degree-$d$ homogeneous polynomials in $n$-variables with complex coefficients. Given a sum 
\begin{align}\label{eq:decomposition}
		\sum_{i=1}^m q_i^d \in S^{2d}(\C^n)
\end{align} 
of powers of quadratic forms $q_1,\ldots, q_m \in S^2(\C^n)$, when can we obtain the addends $q_1^d,\ldots, q_m^d$ up to permutation? This is the \emph{decomposition problem for powers of quadratic forms}. Note that if the $q_i$'s are the quadratic forms corresponding to the covariance matrices $\Sigma_i$'s of the centered Gaussian random vectors $Y_i$'s, then, up to a factor of $ \frac{1}{m} $, \eqref{eq:decomposition} corresponds to the $2d$-th moment of their (uniformly weighted) mixture $Y$. We will fully explain the connection in \Cref{sec:Mixtures-of-Gaussians}. While of course the decomposition problem asks to \emph{obtain} the addends in the sense of algorithmic computation, our set objective in the present work is just to examine when there exists a unique solution. In that case, we say that \textit{identifiability} holds. The smallest nontrivial case (and arguably also the most interesting one) is when moments of degree $6$ are given, that is $d = 3$. Indeed, note that for $d = 2$ we cannot hope for identifiability since $q_1^2 + q_2^2 = \frac{1}{2}(q_1+q_2)^2 + \frac{1}{2}(q_1-q_2)^2$. 

\subsection{Statement of main results}\label{sec:main-result-statement}
In the case of sextics as sums of cubes of quadratic forms, we obtain the following result about general identifiability.

\begin{thm}\label{thm:main-result-cubes-quadratics} 
	Let $ n, m\in \N $ such that one of the following holds:
	\begin{enumerate}\label{eq:main-result-condition}
		\item $n > 16$ and $m\le {n+5 \choose 6}/{n+1 \choose 2}-{n+1 \choose 2}-1$; or
		\item $n \leq 16$ and $m\le {n\choose 2} + 1$.
	\end{enumerate}
	Then, for general $q_1,\ldots, q_m \in S^2(\C^n) $, the sextic $t = \sum_{i=1}^m q_i^3$ has a unique representation as a sum of $ m $ cubes of quadratic forms, up to permutation and third roots of unity. 
\end{thm}

\begin{figure}[h]
	\begin{center}
		\includegraphics[scale=0.35]{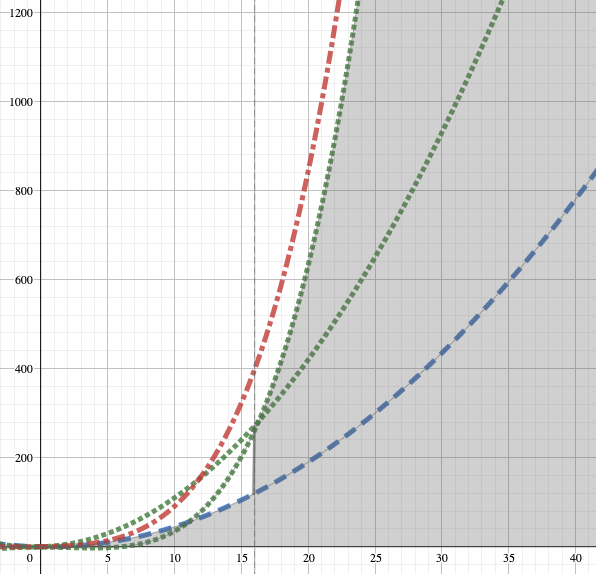}
		\caption{The ranges covered by the main result (\Cref{thm:main-result-cubes-quadratics}) while letting the number of variables $n$ increase along the horizontal axis. The green-dotted curves refer to condition $(1)$: note that the result holds only for $n > 16$ because it requires that ${n+5 \choose 6}/{n+1 \choose 2}-{n+1 \choose 2}-1 > 2\left({{n+1}\choose{2}}-1\right)$ (cfr. \Cref{cor:identifiability_arbitrary}). The blue-dashed line refers to condition $(2)$ (cfr. \Cref{sec:quadratic-identifiability}). The red-dashed-dotted line is the expected number of cubes needed to write a general complex sextic as their sum, i.e., $\left\lfloor \dim S^6(\C^n) / \dim S^2(\C^n) \right\rfloor$: recall that if there are more addends than that number generic identifiability is impossible by dimensionality.}
		\label{fig:result}
	\end{center}
\end{figure}

\noindent
With the connection explained in \Cref{sec:Mixtures-of-Gaussians}, we directly obtain the following consequence about degree-$6$ moments of mixtures of general centered Gaussians. 

\begin{cor}
	\label{thm:main-result-gaussian-mixtures}
	Let $ n, m\in \N $ such that one of the following holds:
	\begin{enumerate}\label{eq:main-result-condition}
		\item $n > 16$ and $m\le {n+5 \choose 6}/{n+1 \choose 2}-{n+1 \choose 2}-1$; or
		\item $n \leq 16$ and $m\le {n\choose 2} + 1$.
	\end{enumerate}
	Let $ Y_1 \sim \mathcal{N}(0, q_1),\ldots, Y_m \sim \mathcal{N}(0, q_m) $ centered normal distributions given by general psd covariance forms $ q_1,\ldots, q_m \in S^2(\R^n)$. Let $ Z_1,\ldots, Z_m $ be any other centered Gaussian random vectors on $ \R^n $ such that one of the following holds:
	\begin{enumerate}[(a)]
		\item the uniformly weighted Gaussian mixtures $ Y = \frac{1}{m} (Y_1 \oplus \ldots \oplus Y_m) $ and $ Z = \frac{1}{m}(Z_1 \oplus \ldots \oplus Z_m) $ agree on the moments of degree $ 6 $; or
		\item for general $ \lambda_1,\ldots, \lambda_m \in \R_{> 0} $ and for $ \mu_1,\ldots, \mu_m\in \R_{\ge 0} $ both summing up to $ 1 $, the Gaussian mixtures $ Y = \lambda_1 Y_1 \oplus \ldots \oplus \lambda_m Y_m $ and $ Z = \mu_1 Z_1 \oplus \ldots \oplus \mu_m Z_m $ agree on the moments of degree $ 6 $ and $ 4 $.
	\end{enumerate}
	Then $ \{Y_1,\ldots, Y_m\} = \{Z_1,\ldots, Z_m\} $ and $ Y = Z $. In case (b), the corresponding mixing weights are equal, too. 
\end{cor}

Condition $(1)$ of \Cref{thm:main-result-cubes-quadratics} may be generalized to guarantee generic identifiability for sums of arbitrary powers of high degree forms (cfr. \Cref{cor:identifiability_arbitrary}). We highlight it here in the case of powers of quadratic forms, which is the case that is relevant for mixtures of centered Gaussians.
\begin{thm}\label{thm:generalD}
	Fix $d \in \N$. For any $n,m \in \N$ such that
	\begin{equation}\label{eq:generalD}
		3{n+1 \choose 2}^2 - 2{n+1 \choose 2} < {n-1+2d \choose 2d}
	\end{equation}
	and 
	\[
		m \leq {n-1+2d \choose 2d}/{n+1 \choose 2}-{n+1 \choose 2}-1
	\]	
	Then, for general $q_1,\ldots, q_m \in S^2(\C^n) $, the degree-$2d$ form $t = \sum_{i=1}^m q_i^d$ has a unique representation as a sum of $ m $ $d$-th powers of quadratic forms, up to permutation and third roots of unity. 
\end{thm}
With the connection explained in \Cref{sec:Mixtures-of-Gaussians}, we obtain an analogous of \Cref{thm:main-result-gaussian-mixtures} for moments of degree $2d$. 

	\begin{figure}[ht]
		\includegraphics[scale=0.6]{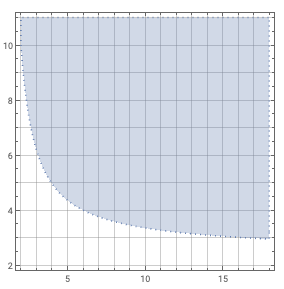}
		\caption{The region of pairs $(n,d)$ satisfying inequality \eqref{eq:generalD} and for which \Cref{thm:generalD} holds. }
		\label{fig:generalD}
	\end{figure}

\subsection{Outline of this paper} 
In \Cref{sec:Mixtures-of-Gaussians}, we give a concise explanation of the connection between the decomposition problem for cubes of quadratics and the moment problem for mixtures of centered Gaussians. In \Cref{sec:prelims} we recall the basic facts on the theory of secant varieties while in \Cref{sec:identifiability} we recall the general identifiability results that are then used to prove the main results. Some parts of our analysis are verified on a computer. The computation is done with the \texttt{Julia} programming language \cite{Julia-2017} and the \texttt{MultivariatePolynomials.jl} library \cite{legat2021multivariatepolynomials}. Our code can be found on GitHub, see \cite{Taveira_Blomenhofer_Code_Cubes_Of_Quadratics_2022}.   

\subsection{Related work} 
On the algorithmic side, Ge, Huang and Kakade provided an algorithm that learns the parameters of mixtures of $ n $-variate Gaussians of rank at most $ \mathcal{O}(\sqrt{n}) $ from their moments of degree at most 6 (see \cite{Ge_Huang_Kakade_2015}). The paper considers a smoothed analysis framework, which is essentially a numerically stable way of saying that the quadratics should be in general position. Mixtures of \emph{centered} Gaussians were recently studied in \cite{Garg_Kayal_Saha_2020}, with a particular focus towards their relation to the decomposition problem for sums of powers of quadratics as well as circuit complexity. The work of Garg, Kayal and Saha \cite{Garg_Kayal_Saha_2020} also studies the more general circuit model of sums of powers of low degree polynomials and was in fact a major motivation for the present work. Sums of powers have been studied from the point of view of algebraic geometry, e.g. \cite{Froeberg_Ottaviani_Shapiro_2012, Lundqvist_Oneto_Reznick_Shapiro_2019}, and recently appeared also in relation to polynomial neural networks \cite{Kileel_Trager_Bruna_2019}.

In the usual terminology of additive decompositions, fixed positive integers $d$ and $k$, the $ (k, d) $-\emph{rank} (or short \emph{rank}) of a degree-$dk$ form is the smallest number of degree-$k$ forms needed to write the given form as their sum of their $ d $-th powers. From an algebro-geometric point of view, additive decompositions of polynomials are studied through secant varieties. In our context, the $m$-th secant variety is the (Zariski) closure of the set of polynomials of degree $dk$ of rank at most $m$. 

Several interesting open questions are yet to be answered for decompositions of homogeneous polynomials as sums of powers. First of all, what is the general rank, i.e., the smallest number of $d$-th powers needed to write a general form of degree $dk$ as their sum? For $k = 1$, i.e., sums of powers of linear forms, it is worth recalling that the question is answered by the celebrated Alexander-Hirschowitz Theorem \cite{hirschowitz1995polynomial} (see also \cite{Brambilla_Ottaviani_2008}). For $k \geq 2$, the complete answer is given in the cases of binary forms and sums of squares in three and four variables in \cite{Lundqvist_Oneto_Reznick_Shapiro_2019}. In general, it is conjectured that the generic rank is as expected by dimension count $\left\lceil {{n+dk-1 \choose dk}}/{{n+k-1\choose k}} \right\rceil$ unless $d = 2$ (cfr. \cite[Conjecture 1.2]{Lundqvist_Oneto_Reznick_Shapiro_2019}). The latter question would be answered by knowing dimensions of all secant varieties.

As for general identifiability, i.e., uniqueness of the decomposition for a general point of a secant variety, a complete answer is given in the case of sums of powers of linear forms ($k=1$): In \cite{Chiantini_Ottaviani_Vannieuwenhoven_2016} it was shown that in all but a few exceptional cases identifiability holds for all subgeneric ranks, while Galuppi \cite{galuppi2019identifiability} completed the classification of cases in which identifiability holds also for generic ranks.  For $k \geq 2$, as far as we know, before the present work, only identifiability for sextics as sums of cubes was recently addressed for rank $ 2 $ \cite{Reznick_2021}. 

Generic identifiability in the range (1) of \Cref{thm:main-result-cubes-quadratics} is proved employing a result by Casarotti and Mella \cite{casarotti2022non}, which translates the study of general identifiablity to the study of dimensions of secant varieties under certain  constraints on rank and dimension. The dimension of the secant varieties of varieties of powers is given by the main result of \cite{nenashev2017note}.  In the range (2) of  \Cref{thm:main-result-cubes-quadratics}, our analysis employs the geometric notion of weak defectivity and tangential contact loci due to Chiantini and Ciliberto \cite{Chiantini_Ciliberto_2002,Chiantini_Ciliberto_2010} and is based on a series of works due to Chiantini, Ottaviani and Vannieuwenhoven \cite{Chiantini_Ottaviani_2012,Chiantini_Ottaviani_Vannieuwenhoven_2014,Chiantini_Ottaviani_Vannieuwenhoven_2016}, where the authors examine the question of \emph{generic tensor identifiability}, i.e., under which conditions does a general tensor of fixed rank $ m $ have a unique decomposition as a sum of $ m $ simple (rank-$1$) tensors. Symmetric tensor decomposition corresponds to the decomposition problem for powers of \emph{linear} forms.

Specifically for cubes of linear forms in $ n $ variables, the generic rank is in $ \theta(n^2) $, whereas efficient algorithms for the decomposition problem succeeding in the smoothed analysis framework are known up to rank $ n $ (e.g. \cite{Anandkumar_Ge_Hsu_Kakade_Telgarsky_2012}, with the original idea dating back to Jennrich, published via Harshman \cite{Harshman_1970}). Various other algorithms for symmetric tensor decomposition of order $ 3 $ exist. Some efficient algorithms can exceed the rank-$ n $ threshold for ``average case'' problems and go up to rank almost $ n^{1.5} $ (e.g. \cite{Ge_Ma_2015}, \cite{Ma_Shi_Steurer_2016}), by relying on the assumption that the rank-$ 1 $ components are drawn from a friendly distribution. Other algorithms can produce decompositions for all subgeneric ranks, but sacrifice  computational efficiency, cf. e.g. the work of Bernardi and Taufer \cite{Bernardi_Taufer_2020}. This leaves a multiplicative gap of $ \theta(n) $ between the regime where generic identifiability holds and the regime where the rank-$ 1 $ components can be efficiently computed in a smoothed analysis framework. 

For quadratic forms, the threshold of $ \mathcal{O}(\sqrt{n}) $ due to Ge, Huang and Kakade \cite{Ge_Huang_Kakade_2015} might not be the final answer either. One might conjecture that efficient algorithms are possible at least as long as the rank is at most the number of variables. Any algorithm succeeding for superquadratic rank $ m \gg \dim S^2(\C^n) $ would by the aforementioned also have nontrivial implications on tensor decomposition, as clearly from a $ 3 $-tensor $ \sum_{i=1}^m q_i^{\ot 3} \in S^3(S^2(\C^n))$ we may compute the $ 6 $-form $ \sum_{i=1}^m q_i^3 $ by applying a linear map. This explains our focus towards finding specific witnesses for quadratic-rank generic identifiability in \Cref{sec:quadratic-identifiability}, although we stress that our results do \emph{not} have algorithmic consequences. 

\section*{Acknowledgements}
We thank Pravesh Kothari for giving inspiring questions and encouraging comments as well as pointing us to the work of Garg, Kayal and Saha \cite{Garg_Kayal_Saha_2020}. We thank Nick Vannieuwenhoven for helpful explanations and both Boris Shapiro and an anonymous reader for telling us about additional, very interesting references. We thank Joseph Landsberg and Laurent Manivel for organizing a "Workshop on geometry and complexity theory" where two of the authors had the opportunity to exchange ideas.


\section{Sums of Powers of Quadratics and Mixtures of Centered Gaussians}\label{sec:Mixtures-of-Gaussians}
\begin{notation*}
	For any field $\F$, let $X = (X_1,\ldots,X_n)$ be a set of variables of $\F^n$. Let $\F[X] = \F[X_1,\ldots,X_n] = \bigoplus_{d \geq 0} \F[X]_d$ be the standard graded polynomial ring where $\F[X]_d$ is identified with $S^d(\F^n)$. 
\end{notation*}

An $ n $-variate Gaussian normal distribution $ \mathcal{N}(\ell, q) $ on $ \R^n $ is given by a pair $ (\ell, q) $ where  $ \ell\in \R[X]_{1} $ is a linear form and  $ q\in \R[X]_{2} $ a positive semi-definite (psd) quadratic form. Some authors require $ q $ to be positive definite, but even if $ q $ has a nontrivial kernel, the pair $ (\ell, q) $ still defines a normal distribution on the affine subspace given by the \emph{mean  vector} $(\ell(e_1),\ldots, \ell(e_n))$ ($e_i$ is the $i$-th coordinate vector) plus the orthogonal complement of the kernel of $ q $ (in the maximal degenerate case, i.e. when $ q = 0 $, this definition gives the Dirac distribution at the mean vector). 

The moments of a random variable $ Y $ can be used to construct a formal power series
\[ 
	\mathbb E [\exp(\langle X,  Y \rangle)] := \sum_{d=0}^\infty \frac{1}{d!}  \mathbb E[\langle X, Y \rangle^d ]\in \R[[X]]
\] 
which is called the \emph{moment generating series} of $ Y $. The expectation $ \mathbb E $ is taken $ X $-coefficient-wise  over the randomness of $ Y $. For the case of a normally distributed random variable $ Y\sim \mathcal{N}(\ell, q) $, this power series takes the very simple and convenient representation
\begin{align}
	\mathbb E [\exp(\langle X,  Y \rangle)] = \exp(\ell + q) = \sum_{d=0}^\infty  \frac{1}{d!}  (\ell + q)^d \label{eq:moment-generating-function-gauss}
\end{align}  
from which we can read all moments by comparing coefficients.  In the case of a centered Gaussian distribution, we have $ \ell = 0 $ and therefore for each $ d\in \N_{>0} $ the moments of degree $ 2d $ are just simply the coefficients of $q^d$. All odd-order moments are zero for a centered Gaussian. 

A \emph{mixture $ Y $ of $ m $ Gaussian} random vectors $ Y_1,\ldots, Y_m $ on $ \R^n $ with mixing weights $ \lambda_1,\ldots, \lambda_m \in \R_{\ge 0}$ satisfying $ \sum_{i=1}^m \lambda_i = 1 $ is a random vector sampled as follows: From a box containing indices $ \{1,\ldots, m\} $, draw the index $ i $ with probability $ \lambda_i $ and then take a sample of $ Y_i $. 
It is easy to see that for any integrable function $ f $ taking values in $ \R^n $, $ \mathbb E[f(Y)] = \sum_{i=1}^m \lambda_i \mathbb E[f(Y_i)] $. 
We therefore suggestively denote $ Y = \lambda_1 Y_1 \oplus \ldots \oplus \lambda_m Y_m $, to remind that on the level of moments and expectations, a mixture random variable is essentially just a convex combination. ``$ \oplus $'' should not be confused with the actual addition on $ \R^n $. Let the Gaussian random vectors be given by pairs of linear and psd quadratic forms $ (\ell_1, q_1),\ldots, (\ell_m, q_m) $. The \emph{moment problem for mixtures of Gaussians} asks to obtain these parameter forms given the moments of the mixture up to a certain degree. In the special case of \emph{centered} Gaussians, i.e. $ \ell_1 = \ldots = \ell_m = 0 $, the degree $ 2d $ moments of $ \lambda_1 Y_1 \oplus \ldots \oplus \lambda_m Y_m $ attain the form 
\begin{align}\label{eq:gm-moments}
	\sum_{i=1}^m \lambda_i q_i^d
\end{align}
for each $ d\in \N_0 $. Since $ \lambda q^d = (\sqrt[d]{\lambda}q)^d $ for each $ \lambda\in \R_{\ge 0}, q\in S^2(\C^n) $, the expression in \Cref{eq:gm-moments} is overparameterized as long as we are only considering moments of one, fixed degree. Let us set the problem in the uniform case, i.e. assume that $ \lambda_1 = \ldots = \lambda_m = \frac{1}{m} $. Then, the problem corresponds to identifiability for $ d $-th powers of \emph{real psd} quadratic forms. 

\begin{remark}\label{rem:psd-zariski-dense}
	Note that the psd quadratic forms are a (Zariski) dense subset of $ S^2(\R^n) $. Since the map $ (q_1,\ldots, q_m) \mapsto \sum_{i=1}^m q_i^d $ is given by rational polynomials, its image when restricted to real points (or even rational points) is (Zariski) dense in its complex image. Therefore, it suffices to show generic identifiability for $ d $-th powers of \emph{complex} quadratic forms.
\end{remark}
 
Note that $ 6 $ (i.e., $d = 3$) is the smallest degree for which we can hope for identifiability. Indeed, the case of sums of squares $(d = 2)$ is not identifiable due to the identity 
\[
	q_1^2 + q_2^2 = \frac{1}{2}(q_1+q_2)^2 + \frac{1}{2}(q_1-q_2)^2.
\]


Before proving our main result \Cref{thm:main-result-cubes-quadratics} on general identifiability of sextics of certain subgeneric ranks as sums of cubes, we show that indeed it allows us to prove our \Cref{thm:main-result-gaussian-mixtures} on identifiability of mixtures of centered Gaussians. We give the reduction for arbitrary $ d $ rather than $ d = 3 $ and remark that it indeed proves \ref{thm:main-result-gaussian-mixtures} 

\begin{thm} 
	Let $ n, m, d\in \N $ such that generic identifiability holds for $ d $-th powers of quadratic forms in $ n $ variables of rank $ m $ (cfr. e.g. \ref{thm:main-result-cubes-quadratics} and \ref{cor:identifiability_arbitrary})
	
	Let $ Y_1 \sim \mathcal{N}(0, q_1),\ldots, Y_m \sim \mathcal{N}(0, q_m) $ centered normal distributions given by general psd covariance forms $ q_1,\ldots, q_m \in S^2(\R^n)$. Let $ Z_1,\ldots, Z_m $ be any other centered Gaussian random vectors on $ \R^n $ such that one of the following holds:
	\begin{enumerate}[(a)]
		\item the uniformly weighted Gaussian mixtures $ Y = \frac{1}{m} (Y_1 \oplus \ldots \oplus Y_m) $ and $ Z = \frac{1}{m}(Z_1 \oplus \ldots \oplus Z_m) $ agree on the moments of degree $ 2d $; or
		\item for general $ \lambda_1,\ldots, \lambda_m \in \R_{> 0} $ and for $ \mu_1,\ldots, \mu_m\in \R_{\ge 0} $ both summing up to $ 1 $, the Gaussian mixtures $ Y = \lambda_1 Y_1 \oplus \ldots \oplus \lambda_m Y_m $ and $ Z = \mu_1 Z_1 \oplus \ldots \oplus \mu_m Z_m $ agree on the moments of degree $ 2d $ and $ 2d-2 $.
	\end{enumerate}
	Then $ \{Y_1,\ldots, Y_m\} = \{Z_1,\ldots, Z_m\} $ and $ Y = Z $. In case (b), the corresponding mixing weights are equal, too. 
\end{thm}


\begin{proof}[Proof of \Cref{thm:main-result-gaussian-mixtures}]
	Let $ p_1,\ldots, p_m \in S^2(\R^n)$ be quadratic psd forms such that
	\[ 
		Z_1\sim \mathcal{N}(0, p_1),\ldots, Z_m\sim \mathcal{N}(0, p_m)
	\]
	
	In case (a), we denote $ \lambda_i := \mu_i := \frac{1}{m} $ for each $ i\in \{1,\ldots, m\} $, while in case (b) we fix $ \lambda_1,\ldots, \lambda_m $ and $ \mu_1,\ldots, \mu_m $ accordingly. Knowing that the degree $ 2d $ moments of $  Y = \lambda_1 Y_1 + \ldots + \lambda_m Y_m $ and $ Z = \mu_1 Z_1 + \ldots + \mu_m Z_m $ are equal, by \Cref{eq:gm-moments} we have 
	\begin{align}\label{eq:main-result-gaussian-mixtures-restatement} 
		\sum_{i=1}^m (\sqrt[d]{\lambda_i}q_i)^d  = \sum_{i=1}^m (\sqrt[d]{\mu_i}p_i)^d 
	\end{align}
	where the quadratic forms $ \sqrt[d]{\lambda_i}q_i $ are general for each $ i\in \{1,\ldots, m\} $. By \Cref{thm:main-result-cubes-quadratics} and \Cref{rem:psd-zariski-dense}, we get 
	\[ 
		\{\sqrt[d]{\lambda_1}q_1,\ldots, \sqrt[d]{\lambda_m}q_m\} = \{\sqrt[d]{\mu_1}p_1,\ldots, \sqrt[d]{\mu_m}p_m\} 
	\]
	Note that for case (a) this is enough to conclude. In case (b), without loss of generality let us assume that $$ \sqrt[d]{\lambda_1}q_1 = \sqrt[d]{\mu_1}p_1  ,\ldots,   \sqrt[d]{\lambda_m}q_m = \sqrt[d]{\mu_m}p_m $$ and write $ \alpha_1 := \sqrt[d]{\frac{\lambda_1}{\mu_1}},\ldots, \alpha_m := \sqrt[d]{\frac{\lambda_m}{\mu_m}}$. Since the degree $ 2d-2 $ moments of $ Y $ and $ Z $ agree, we have 
	\[ 
		\sum_{i=1}^m \lambda_i q_i^{d-1} = \sum_{i=1}^m \mu_i p_i^{d-1}
	\]
	where the $ (d-1) $-st powers of the quadratic forms $ q_1^{d-1},\ldots, q_m^{d-1} $ are linearly independent. Substituting $ p_i = \alpha_i q_i $, we obtain
	\[ 
		\sum_{i=1}^m \lambda_i q_i^{d-1} = \sum_{i=1}^m \mu_i \alpha_i^{d-1} q_i^{d-1}
	\]
	yielding $ \mu_i (\frac{\lambda_i}{\mu_i})^{\frac{d-1}{d}} = \lambda_i $, i.e. $ \mu_i = \lambda_i $ for each $i\in \{1,\ldots, m\} $.
\end{proof}

\subsection{Other types of Gaussian Mixture Problems}
Various other special cases of Gaussian Mixture decomposition problems have been studied. In the case of equal mixing weights, if all $ q_1,\ldots, q_m$ are assumed to be zero, then the problem can be translated into identifiability for special symmetric tensors. For example in \cite{Chiantini_Ottaviani_Vannieuwenhoven_2016} the authors showed generic identifiability from moments of degree $ 3 $ for all  subgeneric ranks. If all quadratics are assumed to be equal (but not necessarily zero), then the problem can still be reduced to tensor decomposition. This is clear from a statistical point of view, but can also be seen algebraically: if $ q:=q_1 = \ldots = q_m $, then the third-order moments attain the form 
\[ 
	\sum_{i=1}^m \ell_i^3 + q\sum_{i=1}^m \ell_i
\]
where the point $ \sum_{i=1}^m \ell_i  = \mathbb E [Y_1 + \ldots + Y_m] $ is known, since it is the vector of first-order moments of the mixture. Thus one can shift the space such that $ \sum_{i=1}^m \ell_i = 0$ and perform classical tensor decomposition on this kind of Gaussian Mixture problem. 

The case of centered Gaussians is significantly more complex and a simple reduction to tensor decomposition is out of reach at the moment, even if Garg, Kayal and Saha \cite[Section 1.3]{Garg_Kayal_Saha_2020} argue that Gaussian Mixtures in full generality might only be a slightly more general class of polynomials than sums of powers of quadratics. Nevertheless, the assumption of centeredness makes the moments easier to handle. For a normally distributed $ Y\sim \mathcal{N}(\ell, q) $ the degree $ 6 $ moment form would otherwise consists of 4 terms. In order to compute them we would have to look at $ k\in \{3,4,5,6\} $ in \Cref{eq:moment-generating-function-gauss}, obtaining
\[ 
	q^3 + {4\choose 2} q^2\ell^2  + {5\choose 1} q \ell^4 +  \ell^6
\]
as the form whose coefficients are the degree $ 6 $ moments of $ Y $.

\section{Preliminaries and general notation on secant varieties}\label{sec:prelims}

After reducing the identifiability problem for mixtures of centered Gaussians to the problem of identifiability of sums of powers of quadratic forms, we recall the basic notations about secant varieties and contact loci that are the basic tools of a geometric approach to the question and will be used in the next section to prove \Cref{thm:main-result-cubes-quadratics}.

\begin{notation*}
	For $ m\in \N_{>0} $ and elements $ v_1,\ldots, v_m $ of a vector space, let $\langle v_1,\ldots, v_m \rangle $ denote the subspace spanned by them. For projective linear subspaces $\P(V_1),\dots,\P(V_m) \subset \P(V)$ we write $\sum_{i=1}^m \P(V_i)$ for the smallest projective linear subspace containing all of them. 
\end{notation*}

\begin{defi} \label{def:secant-variety}
	(Secant variety) Let $ W $ be a variety embedded in an affine or projective space and $ m\in \N_{>0}$. The \emph{$m$-th secant variety} of $W$ is the Zariski-closure of the union of subspaces spanned by $ m $ elements of $ W $, i.e.,
	\[ 
		\sigma_m(W)=\overline{\bigcup_{x_1,\ldots, x_m\in W} \langle x_1,\ldots, x_m \rangle}.
	\]
\end{defi}
\noindent \textbf{Notation.} Recall that the projective space $\P^N$ is the space of equivalence classes of $\C^N \setminus \{0\}$ with respect to the relation which identifies vectors that are one multiple of each other. For any line through the origin $ 0 \in \C^{N+1}$ we associate a projective point $[x] \in \P^N$ where $x$ is a non-zero point of the line. Given a subvariety $ W $ of the projective space $ \P^{N} $, let $ \aff{W} $ denote the \emph{affine cone} of $ W $, which is the set of all representatives of projective points of $ W $ together with the origin, i.e.,
	\[ 
		\aff{W} := \{{0}\} \cup \bigcup_{[x]\in W} \{x\}  \subseteq \C^{N+1}.
	\]
\begin{remark}\label{rem:affine-param-of-secant}
	In terms of affine cones, the secant variety has a convenient parameterization: Let $ W $ and $ m $ be as in \Cref{def:secant-variety}. Then $ \aff{\sigma_m(W)}$ is the closure of the image of 
	\begin{align}\label{eq:affine-param-of-secant}
		\psi_{m, W}\colon \aff{W}^m \to \aff{\sigma_{m}(W)}, (x_1,\ldots, x_m)\mapsto \sum_{i=1}^m x_i.
	\end{align}
\end{remark}


From the latter parametrization, it is clear that in order to hope for generic identifiability  we need a first necessary condition: $\dim \aff{W}^m = \dim \aff{\sigma_m(W)}$. The left-hand side is equal to $m\cdot \dim \aff{W}$ and is clearly an upper bound for the actual dimension of the (cone of) the $m$-th secant variety. It is called \textit{expected dimension} and, whenever it is not attained we say that the variety is \textit{$m$-defective}. The computation of dimensions of secant varieties, and in particular the classification of defective ones, is a difficult challenge in classical algebraic geometry. The following is the main tool to approach the problem, due to Terracini \cite{Te12}, which describes the general tangent space of the secant variety $\sigma_m(V)$ in terms of $m$ general tangent spaces of $W$.
	
\medskip
\begin{notation*}
	Given an affine variety $ W $ and a point $ x\in W $, $ T_x W$ denotes its tangent space at $ x $. If $W$ is a projective variety, embedded in $\P^N$ then, abusing notation, we will also write $T_x W$ for the embedded projective tangent subspace.  
\end{notation*}
 
\begin{lemma}[Terracini's Lemma \cite{Te12}] \label{lem:Terracini}
	 Let $ W $ be a variety and consider for $ m\in \N $ the secant $ \sigma_m(W) $. For general points $ x_1,\ldots, x_m\in W $ and general $ x\in \langle x_1,\ldots, x_m \rangle \subseteq \sigma_m(W) $, we have that  
	\[ 
		T_x \sigma_m(W) = \sum_{i=1}^m T_{x_i}W
	\]
\end{lemma}

Terracini's Lemma gives us a way to determine whether the map \eqref{eq:affine-param-of-secant} is generically finite: a first necessary condition for generic identifiability.

\begin{prop} \label{prop:expected-dimension} Let $ W $ be an irreducible variety and $ m\in \N_0 $ such that for general $ x_1,\ldots, x_m \in W$, the tangent spaces at $ x_1,\ldots, x_m $ are skew, i.e.
	\begin{align}\label{eq:expected-dimension-1}
		\sum_{i=1}^m T_{x_i} \aff{W} = \bigoplus_{i=1}^m T_{x_i} \aff{W}
	\end{align}
	Then the map \eqref{eq:affine-param-of-secant} is generically finite. 
\end{prop}
\begin{proof} 
	Let $ x_1,\ldots, x_m \in W$ be general points. By \Cref{lem:Terracini} and generality of $ x_1,\ldots, x_m $, the left hand side of \Cref{eq:expected-dimension-1} has the dimension of $ \aff{\sigma_m(W)} $, while $ \bigoplus_{i=1}^m T_{x_i}\aff{W} \cong T_{(x_1,\ldots, x_m)}\aff{W}^m  $ has the dimension of $ \aff{W}^m $. Let $x \in \langle x_1,\ldots,x_n \rangle$ be a general point. The fiber dimension formula yields together with Terracini's Lemma (\Cref{lem:Terracini}) and the assumption,
	\[ 
		\dim \psi_{m, W}^{-1}(x) = \dim \aff{W}^m - \dim \aff{\sigma_m(W)} = \dim \bigoplus_{i=1}^m T_{x_i} \aff{W} - \dim \sum_{i=1}^m  T_{x_i} \aff{W}  = 0. 
	\]
\end{proof}



Proposition \ref{prop:expected-dimension} gives a tool to examine whether a mixture decomposition problem has only finitely many solutions, basically only requiring us to calculate the dimension of certain vector spaces. Answering the question of identifiability requires further analysis, since the question is a priori not just about the dimension of the generic fibers of $ \psi_{m, W} $, but about their cardinality. In other words, we need to show that the map $\psi_{m, W}$ is actually birational. If so, we say that both $ \psi_{m, W} $ and $ \sigma_{m}(W) $ are \emph{generically identifiable}.

\medskip
However, in \cite{casarotti2022non}, it is shown that under certain numerical assumptions, proving non $m$-defectivity implies generic $(m-1)$-identifiability. We will employ this fact in Section \ref{sec:nonDef_to_Ident}. 
	
On the other hand, in a series of papers \cite{Chiantini_Ottaviani_2012,Chiantini_Ottaviani_Vannieuwenhoven_2014,
Chiantini_Ottaviani_Vannieuwenhoven_2016}, the geometry of the so-called \textit{tangential contact locus} has been used to study identifiability for tensor decompositions. We will follow the same idea in Section \ref{sec:quadratic-identifiability}. 

\begin{defi}\label{def:contact-locus}
	Let $ W $ a variety, $ m\in \N_{>0} $ and $ x = (x_1,\ldots, x_m)\in W^m $ an $m$-tuple of smooth points of $ W $ with skew tangent spaces. The ($ m $-th) \emph{tangential contact locus} $\mathcal{C}_{W}(x)$ of $ W $ at $ x $ is the projective subvariety of
	\[ 
		\Gamma_W(x) := \left\{ y\in W \mid  T_y W \subseteq \sum_{i=1}^m T_{x_i} W  \right\}.
	\]
constinting of points $z \in \Gamma_W(x)$ such that the irreducible components of $\Gamma_W(x)$ passing through $z$ contain at least one of $x_1,\ldots, x_m$.
\end{defi}

An easy semicontinuity argument enables us to check general identifiability only by studying the tangential contact locus of a specific decomposition $t=x_1+\ldots+x_m$.
Indeed a more general statements holds, see for example \cite[Proposition 2.3]{Chiantini_Ottaviani_Vannieuwenhoven_2014}

\section{Identifiability for powers of forms}\label{sec:identifiability}

Let  $ V_{k, d} = \{q^d \mid q\in \P(S^k(\C^{n})) \}$ denote the projective variety of $ d $-th powers of degree-$ k $ forms. 
We suppress the dependency on $ n\in \N $. 


\begin{prop} \label{prop:iso+tangent}
	For $k, d\in \N $, the map 
	\[
		\iota:\P(S^k(\C^n))\to V_{k, d} \subseteq \P(S^{kd}(\C^n)), p\mapsto p^d
	\] 
	is an embedding. 
\end{prop}
\begin{proof}
	The map $\iota$ can be regarded as the following composition:
	\[
		\P(S^k(\C^n)) \xrightarrow{\nu_d} \P(S^d(S^k(\C^n))) \xrightarrow{\pi_E} \P(S^{kd}(\C^n))
	\]
	where: 
	\begin{itemize}
		\item $\nu_d$ is the $d$-th Veronese embedding sending linear forms to their degree-$d$ power;
		\item $\pi_E$ is the orthogonal linear projection induced by the decomposition \[S^d(S^k(\C^n)) = S^{kd}(\C^n) \oplus E\] where $E$ is the degree-$d$ part of the ideal of the $k$-th Veronese embedding of $\P(\C^n)$. 
	\end{itemize}
	The center of the projection $E$ does not intersect the $2$-nd secant variety $\sigma_2(\nu_d(\P(S^k(\C^n))))$. Indeed, all forms in the latter secant variety is either of the form $\ell^{d-1}m$ or $\ell^d+m^d$, where $\ell,m$ are linear forms, and they are both completely reducible, since $l^d+m^d=\Pi_{i=1}^{d}(l+\zeta^im)$ where $\zeta=e^{\frac{2\pi i}{d}}$ is a primitive $d$-th root of unity. However, the Veronese variety is irreducible and is not contained in any hyperplane. Therefore, $\sigma_2(\nu_d(\P(S^k(\C^n)))) \cap E = \emptyset$ and the projection $\pi_E$ restricted to the image $\nu_d(\P(S^k(\C^n)))$ is an isomorphism and the composition $\pi_E \circ \nu_d$ is an embedding. 
\end{proof}

We prove general identifiability for the $m$-th secant variety of $V_{k,d}$ with two approaches. 

\subsection{From non-defectivity to identifiability}\label{sec:nonDef_to_Ident}
In \cite{casarotti2022non}, Casarotti and Mella derive general identifiability as a consequence of the next-order secant having expected dimension. 

\begin{thm}\cite[Introduction]{casarotti2022non}\label{thm:IdentMella}
	Let $W$ be a smooth variety of dimension $n\in \N$ and let $ m\in \N $. Assume that the $m$-th secant variety is of (expected) dimension $m(n+1)-1$ and $m > 2n$. Then $W$ is $(m - 1)$-identifiable.
\end{thm}
By simple computation of differentials, it is immediate to notice that the tangent space to $V_{k,d}$ at $q = p^d$ is given by
	$ 
		\{ h p^{d-1} \mid h\in S^k(\C^{n}) \}.
	$
	Therefore, by Terracini's Lemma (Lemma \ref{lem:Terracini}), in order to prove that secant varieties of $V_{k,d}$ have the expected dimension, we only have to show that the tangent spaces $T_{p_1^d} V_{k, d},\ldots, T_{p_m^d} V_{k, d}$ are skew for a general choice of the $p_i$'s. That is equivalent to say that the degree-$kd$ part of the ideal $(p_1^{d-1},\ldots,p_m^{d-1})$ has maximal dimension. This is implied by a more general fact related to Fr\"oberg's Conjecture on Hilbert series of general ideals. 
	
	Given a homogeneous ideal $I \subset \C[X]$, the Hilbert series of the associated quotient ring is
	\[
		\HS(\C[X]/I;T) = \sum_{i \in \N}\dim (\C[X]_i / I_i)T^i \in \N[[T]]
	\] 
	where $I_i := I \cap \C[X]_i$. Fr\"oberg's Conjecture \cite{froberg1985inequality} says that given a general ideal $I = (g_1,\ldots,g_m) \subset \C[X]$ with $\deg(g_i) = d_i$, the Hilbert series is given by the  formula
	\begin{equation}\label{eq:Froberg}
		\left[\frac{\prod_{i=1}^m (1-T^{d_i})}{(1-T)^n}\right]
	\end{equation}
	where $[\cdot]$ means that the power series obtained by the fraction is truncated before the first non-positive coefficient. In \cite[Conjecture 2]{nicklasson2017hilbert}, it is conjectured that, whenever $\deg(p_i) > 1$, the ideal $(p_1^{d-1},\ldots,p_m^{d-1})$ has the Hilbert series \eqref{eq:Froberg} for $d_i = (d-1)k$ for a general choice of the $p_i$'s, namely
	\[
		\left[(1-T^{(d-1)k})^m\cdot \sum_{j\in \N} {n-1+j \choose n-1}T^j\right].
	\] 
	If the latter holds, then it is immediate to see that the coefficient of $T^{dk}$ is \[{n-1+kd \choose n-1} - m{n-1+k \choose n-1} = \dim \C[X]_{kd} - m\cdot \dim \C[X]_k,\] i.e., the tangent spaces $T_{p_1^d} V_{k, d},\ldots, T_{p_m^d} V_{k, d}$ are skew. 
	
	For fixed positive integers $a,h$, Nenashev showed in \cite[Theorem 1]{nenashev2017note} that the the coefficient of $T^{a+h}$ of the Hilbert series of an ideal $I$ is as prescribed by \eqref{eq:Froberg} whenever $I = (g_1,\ldots,g_m)$ with $\deg(g_i) = a$, where the $g_i$'s are chosen generically from a nonempty variety $\mathcal{D} \subset S^d(\C^n)$ that is closed under linear transformation, and $m \leq \frac{\dim \C[X]_{a+h}}{\dim \C[X]_h}-\dim \C[X]_h.$ 
	
	In conclusion, by applying the latter result for $a = (d-1)k, h = k$ and $\mathcal{D}$ is the tangential variety of $V_{k,d}$, i.e., $\mathcal{D} = \{p^{d-1}h ~|~ p, h \in S^k\}$, we immediately deduce the following.
	
\begin{thm}
The dimension of the $m$-th secant variety of $V_{k,d}$ is as expected, i.e., \[\dim \sigma_m V_{k,d} = m\cdot \dim \C[X]_k-1\] for $m\leq \frac{\dim \C[X]_{kd}}{\dim \C[X]_k}-\dim \C[X]_k$. 
\end{thm}
Therefore, by Theorem \ref{thm:IdentMella} and since generic $m$-identifiability implies generic $(m-1)$-th identifiability, we have the following identifiability result.
\begin{cor}\label{cor:identifiability_arbitrary}
	The $m$-th secant variety of $V_{k,d}$ is generically $m$-identifiable for $ m \leq \frac{\dim \C[X]_{kd}}{\dim \C[X]_k}-\dim \C[X]_k-1$ provided that $2(\dim \C[X]_k-1) < \frac{\dim \C[X]_{kd}}{\dim \C[X]_k}-\dim \C[X]_k$. 
\end{cor}
In particular, in the $(k,d) = (2,3)$ case, we obtain the condition $(1)$ of the main \Cref{thm:main-result-cubes-quadratics}. Indeed, note that the condition required by \Cref{cor:identifiability_arbitrary}, i.e., \[2\left({n+1 \choose 2}-1\right) < \frac{{n+5 \choose 6}}{{n+1 \choose 2}}-{n+1 \choose 2},\]
holds if and only if $n > 16$.  
\begin{cor}[{\Cref{thm:main-result-cubes-quadratics}, Condition $(1)$}]\label{thm:main-result-cubes-quadratics-restatement}
	Let $ n, m\in \N $ such that $n > 16$. Then, for $m \le {n+5 \choose 6}/{n+1 \choose 2}-{n+1 \choose 2}-1$ and general $q_1,\ldots, q_m \in S^2(\C^n) $ and general $t \in \langle q_1^3,\ldots,q_m^3\rangle$ there is a unique representation of $t$ as sum $m$ cubes of quadratic forms.
\end{cor}
	Corollary \ref{cor:identifiability_arbitrary} provides results also for sums of higher powers of high degree forms. In the case of power of quadratics ($k = 2$), the condition required by \Cref{cor:identifiability_arbitrary} reduces to 
	\begin{equation*}
		3{n+1 \choose 2}^2 - 2{n+1 \choose 2} < {n-1+2d \choose 2d}
	\end{equation*}
and Theorem~\ref{thm:generalD} follows. The latter inequality holds for pairs $(n,d)$ in the region represented in Figure \ref{fig:generalD}. 


\subsection{Tangential contact locus for cubes of quadratics}\label{sec:quadratic-identifiability}
The second approach relies on an earlier result due to Chiantini and Ottaviani \cite{Chiantini_Ottaviani_2012}, which provides generic identifiability as a consequence of a dimension argument for the tangential contact locus. 

A semicontinuity argument enables us to check general identifiability only by studying the tangential contact locus of a specific decomposition $t=x_1+\ldots+x_m$. Indeed, the statements below relate both via another question of dimension.

\begin{prop}[{\cite[Proposition 2.3]{Chiantini_Ottaviani_Vannieuwenhoven_2014}}]\label{prop:contact-locus-curve}
	Let $ W $ be an irreducible, nondegenerate variety of dimension $ n\ge 2 $, which is not $m$-defective. 
	If the generic element of $ \sigma_m(W) $ is not identifiable, then for general $ x\in W^m $ and each $ i\in \{1,\ldots, m\} $, the tangential contact locus to $ W $ at $ x $ must contain a curve through $ x_i $. 
\end{prop}

\begin{thm}[{\cite[Proposition 2.4]{Chiantini_Ottaviani_2012}}]\label{thm:witnesses-generic-identifiability} 
	Let $W$ be a nondegenerate, irreducible smooth variety and $ m\in \N_{>0} $. 
	Consider the following statements:
	\begin{enumerate}[(i)]
		\item The $m$-th secant map $ \psi_{m, W} $ is generically identifiable. 
		\item For every $m$ \emph{general} points $x_1,\ldots, x_m\in W$, $T_{x_1}W,\ldots, T_{x_m} W$ are skew spaces and the dimension of  $ \mathcal{C}_{W}(x_1,\ldots, x_m)$ at every $ x_i $ is zero.
		\item There exist \emph{$m$ specific points} $x_1,\ldots, x_m \in W$ with skew tangent spaces $$ T_{x_1}W,\ldots, T_{x_m} W$$ such that the dimension of  $ \mathcal{C}_{W}(x_1,\ldots, x_m)$ at a specific $ x_i $ is zero.
	\end{enumerate}
	Then we have (iii)$ \implies $(ii)$ \implies $(i).
\end{thm}


In order to prove the idenfifiability result, we will construct a specific set of quadratic forms $q_1,\ldots, q_m $, where $ m ={n\choose 2}+1 $, such that the tangential contact locus at $ q_1,\ldots, q_m $ consists only of the points $ q_1,\ldots, q_m$ with skew tangent spaces. This proves by \Cref{thm:witnesses-generic-identifiability} that the secant of rank up to $ {n\choose 2}+1  $ is identifiable, for \emph{any} $ n \in \N $. We will use variables $ X = (X_1,\ldots, X_n) $ as local affine coordinates. Given $ q_1,\ldots, q_m \in \C[X]_{2} $ we denote, by abuse of notation, with $\aff{\mathcal{C}}(q_1,\ldots, q_m)$ the affine cone of the preimage of the tangential contact locus at the points $[q_1^3],\dots,[q_m^3]$ via the map $\iota$ from \cref{prop:iso+tangent}. Similarly, we denote by $\aff{\Gamma}(q_1,\ldots, q_m) := \aff{\iota^{-1}(\Gamma_{V_{2, 3}}([q_1^3],\dots,[q_m^3]))} $ the preimage via $ \iota $ of $ \Gamma_{V_{2, 3}} $ at $ q_1,\ldots, q_m $, cf. \cref{def:contact-locus}. This notation suppresses the dependency on the variety $ V:= V_{2, 3} $, which depends on the number $ n $ of variables. Making the expression for tangent spaces at $V$ explicit, we have:
\begin{align*}
	\aff{\Gamma}(q_1,\ldots, q_m) &= \left \{ p\in \C[X]_{2} \mid \forall h\in \C[X]_{2}: \exists h_1,\ldots, h_m \in \C[X]_{2}: p^2 h = \sum_{i=1}^m q_i^2 h_i \right \} .	
\end{align*}

\begin{defi}\label{defi:binomial-set} For $i, j\in \{1,\ldots, n\}$ with $ i < j $, define
	\begin{align}
			q_{ij} := (X_i + X_j)^2
	\end{align}
	and let
	\begin{align*}
		\mathcal{B}_n := \{(X_i + X_j)^2 \mid i, j \in  \{1,\ldots, n\}, i < j\} \cup \{X_1^2\}
	\end{align*}
	We call $\mathcal{B}_n$ the \emph{binomial set} of quadratics in dimension $ n $. 
	Up to relabeling we can write $\mathcal{B}_n = \{q_1,\ldots, q_{{n \choose 2}+1}\}$, where the order of the elements is arbitrary. 
\end{defi}

\begin{remark}\label{lem:binomset-induction-stable}  It is an easy consequence of the previous proposition that the following equality holds for $ n\ge 2 $:
	\[ 
		\mathcal{B}_{n-1} \cup \{4X_1^2\} = \{p(X_1,\ldots, X_{n-1}, X_1) \mid p \in \mathcal{B}_n  \}
	\]
\end{remark}

\begin{thm}\label{thm:skewness-binomset} 
	The tangent spaces at elements of the binomial set $ \mathcal{B}_n $ are skew, i.e.
	\[ 
		T_{q_1}\aff{V} + \ldots + T_{q_{{n \choose 2}+1}}\aff{V}  = T_{q_1}\aff{V} \oplus \ldots \oplus T_{q_{{n \choose 2}+1}}\aff{V}
	\]
\end{thm}
\begin{proof}
	We proceed by induction on $ n$. For $n \leq 5$, we verify the statement on a computer. The code may be found on GitHub \cite{Taveira_Blomenhofer_Code_Cubes_Of_Quadratics_2022}. 
	Therefore we can assume that $ n\geq 6 $ and that the claim on $ \mathcal{B}_k $ is true for all $ k < n $. Let $ h_1, h_{ij} \in S^2(\C^n) $, where $ i, j \in \{1,\ldots, n\} $ and $ i < j $. Suppose that
	\begin{align}\label{eq:skewtangents-1}
		0 = h_1 X_1^4 + \sum_{1 \le i < j \le n} h_{ij} (X_i + X_j)^4 	
	\end{align}
	We have to show that $ h_1 = h_{ij} = 0 $ for all $ 1 \le i < j \le n $. Denote $ h_{ji} := h_{ij} $. 
	Since $ \mathcal{B}_n $ is symmetric under permutations of $ \{X_2,\ldots, X_n\} $, we can interchange any two variables not equal to $X_1$. Thus without loss of generality the only two cases to consider are $ (i, j) = (2, 3) $ and $ (i, j) = (1, 2) $. 
	Let us first consider the case $ (i, j) = (2, 3) $. 
	Since $ n\ge 6 $, we may apply the substitution 
	\begin{align*}
		\varphi_{4}\colon \C[X] \to \C [X_1,\ldots, X_3,X_5\ldots, X_n ], X_4\mapsto X_1	
	\end{align*}
	to reduce to a case with one variable less. We obtain 
	\begin{align}\label{eq:skewtangents-2}
	0 = \varphi_{4}(h_1) X_1^4 + \sum_{\substack{1 \le k < l \le n\\
	4\notin \{k, l\}}} \varphi_{4}(h_{kl}) (X_k + X_l)^4 	+ \sum_{k=1}^n \varphi_{4}(h_{4k}) (X_k + X_1)^4 
	\end{align}
	Now 
	note that the form	 $ (X_2 + X_3)^4 $ can only occur on the first summation, yielding $ \varphi_{4}(h_{23}) = 0 $. Therefore by construction $ (X_1 - X_4) $ divides the quadratic form $ h_{23} $. Repeating this same argument with the substitutions
	\begin{align*}
		&\varphi_{5}\colon X_5 \mapsto X_1\\
		&\varphi_{6}\colon X_6 \mapsto X_1	
	\end{align*} 
	yields that $ (X_1 - X_5) $ and $ (X_1 - X_6) $ divide $ h_{23} $, too. Since these linear forms are coprime, $ (X_1 - X_4)(X_1 - X_5)(X_1 - X_6)  $ must divide $ h_{23} $, which for degree reasons is only possible if $ h_{23} = 0$. By symmetry of $ \mathcal{B}_n $, we get that $ h_{ij} = 0$ for all pairs $ \{i, j\} $ not containing 1. Thus \Cref{eq:skewtangents-1} simplifies to 
	\begin{align}\label{eq:skewtangents-3}
		0 = h_1 X_1^4 + \sum_{j = 1}^n h_{1j} (X_1 + X_j)^4 	
	\end{align}
	As for the $(i, j) = (1, 2) $ case: If $ h_{12} $ were not the zero form, then $ h_{12} (X_1 + X_2)^4 $ would contain a monomial of degree at least $ 4 $ in $ X_2 $. Since all other addends in \Cref{eq:skewtangents-3} can only contain monomials of degree at most $ 2 $ in $ X_2 $, the terms of degree at least $ 4 $ in $ X_2 $ from $ h_{12} (X_1 + X_2)^4 $ could not cancel with any other addend from \eqref{eq:skewtangents-3}. After a short argument left to the reader, this forces $ h_{12} = 0 $ and by symmetry thus $ h_{13} = \ldots = h_{1n} = 0 $. Finally, we also must have $ h_{1} = 0 $, as it is the only remaining term in \eqref{eq:skewtangents-3}. 
\end{proof}

Now we show that the tangential contact locus for the binomial set is zero dimensional at each point of the binomial set.

\begin{thm}\label{thm:cl-binomset} For $n\in \N$, and each $ q\in \{q_1,\ldots, q_{{n \choose 2}+1}\} $, locally around $ q $, $ \aff{\mathcal{C}}(q_1,\ldots, q_{{n \choose 2}+1}) $ only contains points from the line  $ \C q $.
\end{thm}
\begin{proof} 
	We use affine notation and proceed by induction on the number $ n $ of variables. The base cases $ n \leq 5$ were verified on a computer, see \cite{Taveira_Blomenhofer_Code_Cubes_Of_Quadratics_2022}. 
	
	Thus let us assume $ n\ge 6 $. As $ \mathcal{B}_n $ is invariant under permutations of $ X_2,\ldots, X_n $, it suffices to show the claim at $ q\in  \{X_1^2, (X_1 + X_2)^2, (X_2 + X_3)^2\} $. In particular, we may assume that $ q $ is a polynomial in $ X_1, X_2, X_3 $. As we work locally around $ q $, it does not matter whether we show the statement for $ \aff{\Gamma} $ or $ \aff{\mathcal{C}} $, cf. \cref{def:contact-locus}.  
	We thus have to show that there exists a neighbourhood $ \mathcal{U}\subseteq \C[X]_{2} $ of $ q $ such that $ \mathcal{U}\cap \aff{\Gamma}(q_1,\ldots, q_m) \subseteq \C q $. 
	Consider the substitution $$ \varphi\colon \C[X] \to \C[X_1,\ldots, X_{n-1}]$$ that maps $X_n$ to $X_1$. Note $ \varphi(q) = q $, as $ n\ge 6 $. By induction hypothesis, we know there exists a neighbourhood $ \mathcal{V} \subseteq \C[X_1,\ldots, X_{n-1}]_{2} $ of $ q $ such that $ \mathcal{V} \cap \aff{\Gamma}(\mathcal{B}_{n-1}) \subseteq \C q $. This means that for all 
	\[
		p\in \varphi^{-1}(\mathcal{V}) \cap \aff{\Gamma}(q_1,\ldots, q_m) 
	\]
	there is $ \lambda\in \C $ such that $ \varphi(p) = \lambda q $. In other words, $$(X_1-X_n)| (p-\lambda q)$$ Repeating the same argument with the substitution $ \varphi' $ that maps $ X_{n-1} $ to $ X_1 $, we obtain another neighbourhood $ \mathcal{V}' $ with the property that for each $$ p\in \varphi'^{-1}(\mathcal{V}') \cap \aff{\Gamma}(q_1,\ldots, q_m) $$ it holds that $ \varphi'(p) = \lambda' q$.
	
	Let $ \mathcal{U} = \varphi^{-1}(\mathcal{V}) \cap \varphi'^{-1}(\mathcal{V}')$, then for each $ p\in \mathcal{U} \cap \aff{\Gamma}(q_1,\ldots, q_m)$,  we can find $ \lambda, \lambda' \in \C $ and linear forms $ \ell, \ell'\in \C[X]_{1} $ such that 
\begin{equation}\label{tang:contact:equality}	
	 \lambda q + \ell (X_1 - X_n) = p = \lambda' q + \ell'(X_1-X_{n-1}) 
\end{equation} 
Finally, we have that $ \ell $ has to be a polynomial in the variables $ \{X_1, X_{n-1}\} $: Indeed, if a variable $ X_j $ for some $ j\notin \{1, n-1\} $ occurred in $ \ell $, then the monomial $ X_jX_n $ on the left hand side of \eqref{tang:contact:equality} could not cancel with any other terms on the left-hand side, but does also not occur on the right hand side. It follows that $ p $ is a polynomial in $ \{X_1, X_2, X_3, X_{n-1}, X_n\} $. Thus we reduced to the case of $ 5 $ variables and proved the claim. 
\end{proof}

\begin{remark}
	\begin{enumerate}
		\item[(a)] Our results do \emph{not} imply that the mixture of cubes of quadratics $\sum_{q\in \mathcal{B}_n} q^3 $ has a unique decomposition as a sum of ${n\choose 2} + 1$ cubes of quadratics! In \cite{Chiantini_Ottaviani_Vannieuwenhoven_2016}, the authors consider some sufficient criteria for the identifiability of specific tensors that maybe, albeit with unnegligible effort, could be transferred to the setting of cubes of quadratics. We did not do any work regarding specific identifiability for cubes of quadratics.
		\item[(b)] We verify \Cref{thm:skewness-binomset} and \Cref{thm:cl-binomset} for $ n=5 $ on a computer. The code is publicly available on GitHub, see \cite{Taveira_Blomenhofer_Code_Cubes_Of_Quadratics_2022}. This base case can  be verified using only methods of Numerical Linear Algebra (such as determining dimensions of certain vector spaces of polynomials) and should therefore be easy to reproduce independently. 
	\end{enumerate}	
\end{remark}




\bibliography{bibML}
\bibliographystyle{plain}

\end{document}